\documentclass[11pt]{amsart}

\usepackage[T1]{fontenc}
\usepackage{lmodern}
\usepackage{microtype}
\usepackage{amsmath,amssymb,amsthm,mathtools,mathrsfs}
\usepackage{enumitem}
\usepackage[hidelinks]{hyperref}

\hypersetup{
  pdftitle={Endpoint resonance in sharp Friedlander--Iwaniec dual sums},
  pdfauthor={Khai-Hoan Nguyen-Dang},
  pdfsubject={Sharp nonlinear dual sums attached to L-functions},
  pdfkeywords={Friedlander--Iwaniec conjecture; nonlinear dual sums; sharp truncation; endpoint resonance; divisor functions; automorphic L-functions}
}

\theoremstyle{plain}
\newtheorem{theorem}{Theorem}[section]
\newtheorem{proposition}[theorem]{Proposition}

\newtheorem{corollary}[theorem]{Corollary}
\newtheorem{conjecture}[theorem]{Conjecture}
\theoremstyle{definition}

\theoremstyle{remark}
\newtheorem{remark}[theorem]{Remark}

\newcommand{\C}{\mathbb C}
\newcommand{\R}{\mathbb R}
\newcommand{\Z}{\mathbb Z}
\newcommand{\eps}{\varepsilon}
\newcommand{\e}{\mathrm e}
\newcommand{\Bcal}{\mathcal B}
\newcommand{\Res}{\operatorname*{Res}}
\newcommand{\distT}{\operatorname{dist}_{2\pi}}
\newcommand{\doi}[1]{\href{https://doi.org/#1}{\nolinkurl{doi:#1}}}
\numberwithin{equation}{section}
\allowdisplaybreaks
\title[Counterexamples of  Friedlander--Iwaniec dual sums conjecture]
{Counterexamples of Friedlander--Iwaniec dual sums conjecture}

\author{Khai-Hoan Nguyen-Dang}
\address{Morningside Center of Mathematics, Chinese Academy of Sciences, No.~55 Zhongguancun East Road, Beijing 100190, China}
\email{khaihoann@gmail.com}


\begin{document}

\begin{abstract}
Let $a(n)$ and $b(n)$ are arithmetic sequences, and
\[
  A(s)=\sum_{n\ge1}a(n)n^{-s},
  \qquad
  B(s)=\sum_{n\ge1}b(n)n^{-s},
\]
be the two Dirichlet series related by a certain
functional equation. Let $m$ be the \emph{analytic degree} of the functional equation.  For \(x>0\) and a positive integer \(N\), Friedlander and Iwaniec (2005) define the sharply truncated nonlinear dual sum
\[
  \mathcal B_{\ell,D}(x,N)
  :=
  \sum_{\substack{n\in\mathbb N\\ n\le N}}
  b(n)n^{-\beta_m}
  \cos\!\left(
    2\pi m\left(\frac{nx}{D}\right)^{1/m}
    +\frac{\pi\ell}{4}
  \right),
\]
where
\(D\ge1\) is the conductor, \(\beta_m:=\frac{m+1}{2m}\), and \(\ell=m-3-2k\)
is determined by the archimedean weight \(k\) of the functional equation. 

Their Conjecture~1 predicts that, for every \(\varepsilon>0\),
\[
  \mathcal B_{\ell,D}(x,N)
  \ll_{\varepsilon,\boldsymbol\kappa}
  (DNx)^\varepsilon,
\]
uniformly in the variables \(x\) and \(N\), with the degree, conductor, and archimedean datum fixed.  We give counterexamples to this prediction with
\[
  A(s)=B(s)=\zeta(s)^m,\; m\geq 4
\]
where
\[
  \zeta(s):=\sum_{n\ge1}n^{-s}
  \qquad(\operatorname{Re}s>1)
\]
is the Riemann zeta function.
\end{abstract}

\maketitle

\section{Introduction}

\subsection{Arithmetic sums and Dirichlet series}
Let
\begin{equation}\label{eq:basic-data-intro}
  S_a(x):=\sum_{n\le x}a(n),
  \qquad
  A(s):=\sum_{n\ge1}a(n)n^{-s}.
\end{equation}
A fundamental problem of number theory is to recover precise information about the sumfunction $S_a(x)$ from analytic properties of its Dirichlet series $A(s)$. For example, under certain coefficient hypotheses Perron formula gives (see \cite[Lemma~3.12]{Titchmarsh}), with $\alpha=1+\varepsilon$ and $1\le T\le x$,
\begin{equation}\label{eq:Perron-intro}
  S_a(x)
  =\frac{1}{2\pi i}
   \int_{\alpha-iT}^{\alpha+iT}
     A(s)\frac{x^s}{s}\,ds
   +O_\varepsilon\!\left(\frac{x^{1+2\varepsilon}}{T}\right).
\end{equation}

Suppose moreover that $A(s)$ has analytic continuation to a region to the left of $\Re s=\alpha$ and has sufficient growth control on vertical lines.  One closes a rectangle and moves the vertical side from $\Re s=\alpha$ to a line $\Re s=c<\alpha$.  The residue theorem gives
\begin{align}
  S_a(x)
  &=\sum_{\rho\ \mathrm{crossed}}
   \Res_{s=\rho}\!\left(A(s)\frac{x^s}{s}\right)
   +\frac{1}{2\pi i}\int_{c-iT}^{c+iT}
      A(s)\frac{x^s}{s}\,ds \notag\\
  &\quad+\text{controlled truncation terms}.
  \label{eq:contour-shift-schematic}
\end{align}
If $\rho\ne0$ is a pole of $A(s)$ of order $r$, its residue in \eqref{eq:contour-shift-schematic} has the form
\[
  x^\rho P_{\rho,r-1}(\log x),
\]
where $P_{\rho,r-1}$ is a polynomial of degree at most $r-1$. For example, the pole of the kernel $1/s$ at $s=0$ is treated at \cite[(2.2)--(2.3)]{FriedlanderIwaniec}. The authors shift from $\Re s=1+\varepsilon$ to $\Re s=-\varepsilon$, the crossed poles are the pole of $A(s)$ at $s=1$ and the kernel pole at $s=0$.

\subsection{Functional equations, duality, and stationary phase}

In \cite{FriedlanderIwaniec}, Friedlander and Iwaniec assume that, for $\operatorname{Re}s>1$,
\begin{equation}\label{eq:FE-intro}
  A(1-s)=w\gamma(s)B(s),
  \qquad |w|=1,
  \qquad
  B(s)=\sum_{n\ge1}b(n)n^{-s},
\end{equation}
where $B(s)$ converges absolutely for $\Re s>1$ and $\gamma(s)$ is an archimedean gamma quotient.  The complete hypotheses are recalled in Section~\ref{sec:zeta}.  

After the contour shift, the change of variable $s\mapsto1-s$ and \eqref{eq:FE-intro} return the remaining integral to the line $\Re s=\alpha=1+\varepsilon$:
\begin{equation}\label{eq:dual-integral-intro}
  I(x)
  =\frac{w}{2\pi i}
   \int_{\alpha-iT}^{\alpha+iT}
     \gamma(s)B(s)\frac{x^{1-s}}{1-s}\,ds.
\end{equation}
Since $B(s)$ is absolutely convergent there, termwise integration is legitimate and yields (see equations~(2.3)--(2.5) in \cite[p.~497]{FriedlanderIwaniec})
\begin{equation}\label{eq:dual-kernel-intro}
  I(x)=wx\sum_{n\ge1}b(n)c_T(nx),
  \qquad
  c_T(y):=\frac{1}{2\pi i}
  \int_{\alpha-iT}^{\alpha+iT}
       \frac{\gamma(s)}{1-s}y^{-s}\,ds.
\end{equation}
We call \eqref{eq:dual-kernel-intro} a \emph{dual expression} because the original coefficients $a(n)$ have been replaced by the coefficients $b(n)$ paired with them by the functional equation, and the original sharp step has been replaced by the gamma-factor transform $c_T$.  In the standard automorphic situation, $B(s)$ is the $L$-series attached to the representation-theoretic dual; for the classical $\mathrm{GL}_3$ formulation (see Miller--Schmid \cite[(5.5)]{MillerSchmid}).  

The preceding argument belongs to the broad duality philosophy of Poisson and
Voronoi summation. Poisson summation underlies the theta transformation, the functional equation of $\zeta(s)$, lattice-point problems, and completion of finite exponential sums;
see \cite[\S4.3]{IwaniecKowalski} and
\cite[(1.1), \S2]{MillerSchmid}. Voronoi summation is an arithmetic transform whose kernel is governed by the
functional equation of an $L$-function.  See Vorono\"i's original memoirs
\cite{Voronoi1904a,Voronoi1904b} and the modern discussion
\cite[\S3]{MillerSchmid}. The relation between multiple gamma factors and summation
formulae was developed by Chandrasekharan and Narasimhan
\cite{ChandrasekharanNarasimhan}.  For $r\ge1$, $\mathrm{GL}_r$ denotes
the general linear group of invertible $r\times r$ matrices.  For automorphic
formulae on $\mathrm{GL}_3$ and $\mathrm{GL}_n$, see
\cite{MillerSchmidGL3,MillerSchmidGLn}.

\subsection{The Friedlander--Iwaniec sharp sum}

Put
\[
  \theta_m:=\frac{m-1}{2m},
  \qquad
  \beta_m:=\frac{m+1}{2m}=1-\theta_m,
\]
and define
\begin{equation}\label{eq:FI-B-intro}
  \Bcal_{\ell,D}(x,N)
  :=\sum_{n\le N}b(n)n^{-\beta_m}
  \cos\!\left(
    2\pi m\left(\frac{nx}{D}\right)^{1/m}
    +\frac{\pi\ell}{4}
  \right).
\end{equation}
Let $\boldsymbol\kappa:=(\kappa_1,\ldots,\kappa_m)$ denote the
archimedean parameter array, recalled precisely in \eqref{eq:FI-gamma}.  For
the real parameter arrays relevant to this paper, the one-cosine form extracted
in the proof of \cite[Theorem~1.2]{FriedlanderIwaniec} is
\begin{align}
  \sum_{n\le x}a(n)
  &=R_A(x)
   +\frac{w}{\pi\sqrt m}D^{1/(2m)}x^{\theta_m}
    \Bcal_{\ell,D}(x,N) \notag\\
  &\quad+O_{\varepsilon,\boldsymbol\kappa}\!\left(
      D^{1/m}N^{-1/m}x^{(m-1)/m+\varepsilon}
    \right),
  \label{eq:FI-formula-intro}
\end{align}
where
\[
  R_A(x):=\Res_{s=1}\!\left(A(s)\frac{x^s}{s}\right),
  \qquad x\ge D^{1/2},
  \qquad 1\le N\le x.
\]

Friedlander and Iwaniec formulate the following conjecture.

\begin{conjecture}[Friedlander--Iwaniec Conjecture~1]\label{conj:FI1}
For every $\varepsilon>0$ and every automorphic series $B(s)$ in their
framework, with the automorphic datum fixed,
\begin{equation}\label{eq:FI-conj-intro}
  \Bcal_{\ell,D}(x,N)
  \ll_{\varepsilon,\boldsymbol\kappa}(DNx)^\varepsilon.
\end{equation}
\end{conjecture}

We explain why the conjecture is hard at first glance. The assumed coefficient bound $b(n)\ll_\eta n^\eta$ gives only
\begin{equation}\label{eq:FI-trivial-bound}
  |\Bcal_{\ell,D}(x,N)|
  \le\sum_{n\le N}|b(n)|n^{-\beta_m}
  \ll_\eta N^{1-\beta_m+\eta}
  =N^{\theta_m+\eta}.
\end{equation}
Conjecture~\ref{conj:FI1} asks that the entire positive power
$N^{\theta_m}$ be replaced by a subpower factor.  A heuristic provides limited motivation. For example, if
$0<\eta<1/(2m)$, then
\[
  \sum_{n\ge1}|b(n)|^2n^{-2\beta_m}
  \ll_\eta\sum_{n\ge1}n^{-1-1/m+2\eta}<\infty.
\]

Friedlander and Iwaniec stress that their method cannot improve the trivial
bound for the dual sum because the functional equation is involutory: applying the duality twice returns to the original series. In particular, they treat the completely factorizable degree-three datum
\[
  A(s)=L(s,\chi_1)L(s,\chi_2)L(s,\chi_3),
  \qquad D=D_1D_2D_3,
\]
where the $\chi_j$ are primitive Dirichlet characters. They prove
\begin{equation}\label{eq:FI-degree-three-B-progress}
  \Bcal_{\ell,D}(x,N)
  \ll
  \bigl(D^{20}N^{11}x\bigr)^{1/42}(\log x)^3;
\end{equation}
see \cite[(4.4)--(4.9)]{FriedlanderIwaniec}.  This saves a power of $N$ over
the trivial estimate in suitable parameter ranges, but is far from a
subpower bound.

The same nonlinear phase occurs in the \emph{standard twist}
\[
  F(s,\alpha)
  :=\sum_{n\ge1}a(n)e(-\alpha n^{1/m})n^{-s}.
\]
Kaczorowski and Perelli establish its meromorphic continuation, a functional
equation of a new type, and a detailed description of its poles and residues
\cite[(1.1), Theorems~2--3]{KaczorowskiPerelli}.  These results clarify the analytic structure of
the nonlinear twist but do not estimate every sharp partial sum in
Conjecture~\ref{conj:FI1}.  Lin and Sun obtain nontrivial estimates for smoothly weighted nonlinear
twists attached to $\mathrm{GL}_3\times\mathrm{GL}_2$ and derive sharp-cut
consequences in specified parameter ranges
\cite[Theorem~1.1 and Corollary~1.2]{LinSun}.  These results are
substantial, but they do not furnish a subpower estimate uniformly as required by Conjecture~\ref{conj:FI1}.

\subsection{The method and main results}

We test Conjecture~\ref{conj:FI1} on
\[
  A(s)=B(s)=\zeta(s)^m.
\]
Its coefficients are the $m$-fold divisor numbers
\begin{equation}\label{eq:dm-definition-intro}
  d_m(n)
  :=\#\{(n_1,\ldots,n_m)\in\mathbb N^m:n_1\cdots n_m=n\},
\end{equation}
so $d_m(n)\ge1$.  With
\[
  \phi_m:=\frac{\pi(m-3)}4,
\]
the corresponding sharp sum is
\begin{equation}\label{eq:Bm-intro}
  \Bcal_m(x,M)
  :=\sum_{n\le M}d_m(n)n^{-\beta_m}
  \cos\!\left(2\pi m(nx)^{1/m}+\phi_m\right).
\end{equation}

Starting from any prescribed scale $x\asymp \lambda N^a$, we perturb $x$ by a relative amount of order $N^{-(a+1)/m}$ so that the phase of the final term $n=N$ is exactly at a maximum of the cosine. Near this endpoint, the phase changes on the scale 
\[ H\asymp N^{1-(a+1)/m}. \] 
The last $H$ terms remain in one fixed positive arc. Since $d_m(n)\ge1$, their total weighted mass is 
\[ HN^{-(m+1)/(2m)} \asymp N^{(m-3-2a)/(2m)}. \]
This terminal block is exactly the difference of the two sharp partial sums at $N-H$ and $N$. The following theorem records this endpoint mechanism with its precise scale and uniformity.

\begin{theorem}[Sharp endpoint lower bound]\label{thm:main}
Let $m\ge2$ be an integer, let $\lambda>0$, and let $0<a<m-1$.  For every
sufficiently large integer $N$, there exist $x_N>0$, an integer $H_N$ with
$1\le H_N<N/2$, and
\[
  M_N\in\{N-H_N,N\}
\]
such that
\begin{align}
  x_N
  &=\lambda N^a\left(1+O_{m,a,\lambda}
    \bigl(N^{-(a+1)/m}\bigr)\right),\label{eq:main-x}\\
  H_N
  &\asymp_{m,a,\lambda}N^{1-(a+1)/m},\label{eq:main-H}
\end{align}
and
\begin{equation}\label{eq:main-lower}
  |\Bcal_m(x_N,M_N)|
  \gg_{m,a,\lambda}N^{\Delta_m(a)},
\end{equation}
where $\Delta_m(a):=\frac{m-3-2a}{2m}$. More precisely,
\begin{equation}\label{eq:main-block}
  |\Bcal_m(x_N,N)-\Bcal_m(x_N,N-H_N)|
  \gg_{m,a,\lambda}N^{\Delta_m(a)}.
\end{equation}
\end{theorem}

A positive value of $\Delta_m(a)$ gives a power-sized obstruction to any uniform subpower estimate, hence produces counterexamples for the conjecture~\ref{conj:FI1}.

\begin{corollary}\label{cor:consequences}
The following statements hold.
\begin{enumerate}[label=\textup{(\roman*)},leftmargin=2.5em]
\item Let $m\ge4$ and
\begin{equation}\label{eq:a-power}
  0<a<\frac{m-3}{2}.
\end{equation}
Then \eqref{eq:FI-conj-intro} fails for $B(s)=\zeta(s)^m$ along a sequence
with $x_N\asymp N^a$.  

\item Let $m\ge6$ and $\lambda>1$.  One may take $a=1$, in which case
\[
  1\le N-H_N<N\le x_N
\]
for all sufficiently large $N$, and
\begin{equation}\label{eq:linear-lower-intro}
  |\Bcal_m(x_N,M_N)|
  \gg_{m,\lambda}N^{(m-5)/(2m)}.
\end{equation}
Thus the Conjecture~1 fails even in the printed
range $1\le N\le x$. 

\item Let $m\ge5$.  One may take
\[
  a=\frac2{m-1},
  \qquad
  \lambda=1,
\]
so that
\[
  N=x_N^{(m-1)/2}(1+o(1))
\]
and
\begin{equation}\label{eq:balancing-lower-intro}
  |\Bcal_m(x_N,M_N)|
  \gg_m N^{\frac{m^2-4m-1}{2m(m-1)}}.
\end{equation}
The exponent is positive precisely for integral $m\ge5$.  
\end{enumerate}
\end{corollary}

\begin{remark}[Range law]\label{rem:range-law}
For $m>3$, writing $N\asymp x^\rho$ amounts to taking $a=1/\rho$.  The
exponent in \eqref{eq:main-lower} is positive exactly when \(\rho>\frac2{m-3}.\)
\end{remark}

\begin{remark}
The constant shown in the printed display (1.13) is instead
$w(\pi m)^{-1/2}D^{1/(2m)}$.  The last formula on p.~499, gives \( w\left(\frac{2Q}{\pi m}\right)^{1/2} =\frac{wD^{1/(2m)}}{\pi\sqrt m}. \) Thus the printed statement contains an extraneous factor $\sqrt\pi$.  The
same discrepancy persists in the printed degree-two specialization, and the
typo is recorded explicitly in \cite[p.~835, footnote~2]{Leung}.  This fixed
nonzero factor has no effect on the sharp-sum conjecture or on any exponent
below.  Clozel and Sarnak further explain that a general archimedean analysis
may carry two leading constants rather than a single cosine; see
\cite[Hypothesis~2.1, Lemma~2.1, and (2.30)]{ClozelSarnak}.  For the datum
$\zeta(s)^m$, all gamma parameters are real and zero, so the cosine form is
valid.  The two-frequency version is treated in Remark~\ref{rem:CS-scope}.
\end{remark}

\subsection{Remarks on Friedlander--Iwaniec Conjecture~2}
\label{subsec:importance-C2}

Set
\[
  E_A(x)
  :=
  \sum_{n\le x}a(n)-R_A(x),
  \qquad
  R_A(x)
  :=
  \Res_{s=1}\left(A(s)\frac{x^s}{s}\right).
\]
Friedlander and Iwaniec formulate the following conjectural bound.

\begin{conjecture}[Friedlander--Iwaniec Conjecture~2]
\label{conj:FI2}
Let $A(s)$ be an automorphic series of degree $m$ and conductor $D$
in the Friedlander--Iwaniec framework. Then, for every
$\varepsilon>0$,
\[
  E_A(x)
  \ll_{\varepsilon,\boldsymbol\kappa}
  D^{1/(2m)}x^{(m-1)/(2m)+\varepsilon},
\]
where $\boldsymbol\kappa=(\kappa_1,\ldots,\kappa_m)$ is the
archimedean parameter array.
\end{conjecture}

Friedlander and Iwaniec note that even the Riemann hypothesis for
$A(s)$ gives only
\[
  E_A(x)\ll_\varepsilon x^{1/2}(Dx)^\varepsilon;
\]
see \cite[p.~500]{FriedlanderIwaniec}. They propose to deduce
Conjecture~\ref{conj:FI2} from Conjecture~\ref{conj:FI1}, after
relaxing the printed condition $N\le x$, by choosing
\[
  N\asymp D^{1/2}x^{(m-1)/2}.
\]
Without additional arithmetic input, their functional-equation
argument gives
\[
  E_A(x)
  \ll_{\varepsilon,\boldsymbol\kappa}
  D^{1/(m+1)}x^{(m-1)/(m+1)+\varepsilon},
  \qquad x\ge D^{1/2},
\]
while for
\[
  A(s)=L(s,\chi_1)L(s,\chi_2)L(s,\chi_3),
  \qquad D=D_1D_2D_3,
\]
they prove
\[
  E_A(x)\ll_\varepsilon D^{38/75}x^{37/75+\varepsilon};
\]
see \cite[Proposition~1.1 and Theorem~4.2]{FriedlanderIwaniec}.
The counterexamples in the present paper invalidate the proposed
uniform sharp-dual-sum route to Conjecture~\ref{conj:FI2}, but do not
themselves disprove Conjecture~\ref{conj:FI2}. For $A(s)=\zeta(s)^m$, one has
\[
  R_A(x)=xP_{m-1}(\log x),
\]
with $P_{m-1}$ a polynomial of degree $m-1$, and
Conjecture~\ref{conj:FI2} becomes the Piltz divisor conjecture
\[
  \Delta_m(x)
  :=
  \sum_{n\le x}d_m(n)-xP_{m-1}(\log x)
  \ll_{m,\varepsilon}
  x^{(m-1)/(2m)+\varepsilon}.
\]
For $m=2$, Hardy proved
\[
  \Delta_2(x)\ne o(x^{1/4}),
\]
so the exponent $1/4$ cannot be decreased \cite{Hardy}. Huxley proved
\[
  \Delta_2(x)\ll_\varepsilon x^{131/416+\varepsilon},
  \qquad
  \frac{131}{416}=0.314903\ldots,
\]
\cite{Huxley}, and the preprint of Li and Yang announces
\[
  \Delta_2(x)
  \ll_\varepsilon
  x^{0.3144831759741\ldots+\varepsilon};
\]
see \cite[Theorem~1.2]{LiYang}. For $m=3$, Kolesnik proved
\[
  \Delta_3(x)\ll_\varepsilon x^{43/96+\varepsilon},
\]
whereas the conjectural exponent is $1/3$; see \cite{Kolesnik} and
\cite[p.~504]{FriedlanderIwaniec}.

There is a parallel degree-two cuspidal problem. For the normalized
Hecke eigenvalues $\lambda_f(n)$ of a holomorphic Hecke newform,
Conjecture~\ref{conj:FI2} predicts
\[
  \sum_{n\le x}\lambda_f(n)
  \ll_{f,\varepsilon}x^{1/4+\varepsilon}.
\]
In the newform setting treated by Wu,
\[
  \sum_{n\le x}\lambda_f(n)
  \ll_f
  x^{1/3}(\log x)^{\rho_{1/2}^{+}},
  \qquad
  \rho_{1/2}^{+}=-0.118515\ldots;
\]
see \cite[Theorem~2]{Wu}. Conversely, Hafner and Ivi\'c establish,
for the cusp forms considered there and some $c_f>0$,
\[
  \sum_{n\le x}\lambda_f(n)
  =
  \Omega_\pm\!\left(
    x^{1/4}
    \exp\!\left\{
      c_f\frac{(\log_2x)^{1/4}}{(\log_3x)^{3/4}}
    \right\}
  \right);
\]
see \cite[Theorem~2]{HafnerIvic}. Thus $1/4$ is the correct
power scale for lower bounds, although the corresponding pointwise
upper bound remains open.

As a higher-degree example, let $\pi$ be a $\mathrm{GL}_3$
Hecke--Maass cusp form and let $f$ be a $\mathrm{GL}_2$ holomorphic
or Maass cusp form. Under the Ramanujan--Petersson conjecture,
Lin and Sun prove
\[
  \sum_{r^2n\le x}
  \lambda_\pi(r,n)\lambda_f(n)
  \ll_{\pi,f,\varepsilon}
  x^{5/7-1/364+\varepsilon};
\]
see \cite[Corollary~1.4]{LinSun}. This saves a power over the general
Friedlander--Iwaniec exponent $5/7$, while
Conjecture~\ref{conj:FI2} predicts the exponent $5/12$.

\subsection{Organiztion}
Section~\ref{sec:sector} proves a general endpoint-sector theorem.
Section~\ref{sec:zeta} verifies the Friedlander--Iwaniec datum for
$\zeta(s)^m$ and proves Theorem~\ref{thm:main} and
Corollary~\ref{cor:consequences}.  
\subsection{Notation and convention}
Throughout the paper,
\[
  \mathbb N:=\{1,2,\ldots\},
\]
and every sum over $n\le y$ is over positive integers.  An \emph{arithmetic
sequence} is a function $f:\mathbb N\to\mathbb C$.   If $G\ge0$, then
$F\ll_{\mathcal P}G$, equivalently $F=O_{\mathcal P}(G)$, means that
$|F|\le C_{\mathcal P}G$ for a constant depending only on the displayed fixed
parameters $\mathcal P$.  We write $F\asymp_{\mathcal P}G$ when both
$F\ll_{\mathcal P}G$ and $G\ll_{\mathcal P}F$ hold, $F=o(G)$ when
$F/G\to0$, and $F\sim G$ when $F/G\to1$.  Unless another limit is specified,
asymptotic statements involving $N$ are as $N\to\infty$.  A quantity is
\emph{power-sized} if it is $\gg N^\delta$ for some fixed $\delta>0$, and is
\emph{subpower} if it is $O_\varepsilon(N^\varepsilon)$ for every
$\varepsilon>0$.  If $G(x)>0$ for all sufficiently large $x$, then
$F=\Omega_+(G)$ means $\limsup_{x\to\infty}F(x)/G(x)>0$,
$F=\Omega_-(G)$ means $\liminf_{x\to\infty}F(x)/G(x)<0$, and
$F=\Omega_\pm(G)$ means that both assertions hold.  We write $\log_j$ for
the $j$-fold iterated natural logarithm.

\section{A sharp-endpoint sector lemma}\label{sec:sector}

For $u,v\in\mathbb R$, write
\[
  \operatorname{dist}_{2\pi}(u,v)
  :=
  \min_{j\in\mathbb Z}|u-v-2\pi j|.
\]
Let $\Phi_x:[1,\infty)\to\mathbb R$ be a real-valued phase depending
on a parameter $x>0$, and let
$\Omega:\mathbb R\to\mathbb C$ be a continuous
$2\pi$-periodic function. We consider a sum of the
form
\[
  S_\Omega(x,M)
  =
  \sum_{n\le M}c_nw_n\Omega(\Phi_x(n)),
\]
the factor $c_nw_n$ is a certain arithmetic input.

The following observation isolates the elementary mechanism behind all of our counterexamples. The phase \[ \Phi_x(n):=\kappa(nx)^\nu \] varies slowly near a large endpoint because \(0<\nu<1\). We choose \(x\) so that the terminal phase \(\Phi_x(N)\) falls at a point \(t_0\pmod{2\pi}\) where the periodic profile \(\Omega\) does not vanish. By continuity, after one fixed rotation in the complex plane, \(\Omega(\Phi_x(n))\) then has uniformly positive real part as long as \(\Phi_x(n)\) remains in a sufficiently short arc around \(t_0\). If \(x\asymp N^a\), then, for \(n\asymp N\), \[ \Phi_x'(n) =\kappa\nu x^\nu n^{\nu-1} \asymp N^{\nu(a+1)-1}. \] Hence the phase changes by only \(O(1)\) on a terminal interval of the reciprocal first-derivative length \[ H\asymp \frac{1}{|\Phi_x'(N)|} \asymp N^{1-\nu(a+1)}. \] The condition \[ -1<a<\nu^{-1}-1 \] has a transparent meaning: its lower bound ensures that the resonant frequency tends to infinity, while its upper bound ensures that \(H\to\infty\); together they also give \(H=o(N)\). Because the coefficients are bounded below by a positive constant, the coherent block has weighted mass \[ \sum_{N-H<n\le N}c_n n^{-\beta} \asymp HN^{-\beta} \asymp N^{1-\beta-\nu(a+1)}. \] This block is exactly the difference between the two sharp partial sums at \(N-H\) and \(N\). Thus at least one of the two endpoint partial sums must be as large as the coherent block. The next theorem records this principle in a form independent of the particular Friedlander--Iwaniec phase.

\begin{theorem}\label{thm:sector}
Let $0<\nu<1$, let $\beta\in\R$, and let $\kappa>0$.  Let $\Omega:\R\to\C$ be a continuous, $2\pi$-periodic function that is not identically zero, and let $(c_n)_{n\ge1}$ be real numbers satisfying
\begin{equation}\label{eq:coeff-lower}
  c_n\ge c_*>0\qquad(n\ge1).
\end{equation}
For $x>0$ and $M\ge0$, define
\begin{equation}\label{eq:general-S}
  S_{\Omega}(x,M)
  :=\sum_{n\le M}c_n n^{-\beta}
  \Omega\!\left(\kappa(nx)^\nu\right).
\end{equation}
Fix $\lambda>0$ and
\begin{equation}\label{eq:a-general-range}
  -1<a<\nu^{-1}-1.
\end{equation}
Then, for every sufficiently large integer $N$, there exist $x_N>0$ and an integer $H_N$ with $1\le H_N<N/2$ such that
\begin{align}
  x_N
  &=\lambda N^a\left(1+O\bigl(N^{-\nu(a+1)}\bigr)\right),\label{eq:sector-x}\\
  H_N
  &\asymp N^{1-\nu(a+1)},\label{eq:sector-H}
\end{align}
and
\begin{equation}\label{eq:sector-block}
  \bigl|S_{\Omega}(x_N,N)-S_{\Omega}(x_N,N-H_N)\bigr|
  \gg N^{1-\beta-\nu(a+1)}.
\end{equation}
Consequently, for some $M_N\in\{N-H_N,N\}$,
\begin{equation}\label{eq:sector-max}
  |S_{\Omega}(x_N,M_N)|
  \gg N^{1-\beta-\nu(a+1)}.
\end{equation}
All implied constants in this theorem may depend on the fixed data $\nu,\beta,\kappa,\lambda,a,\Omega$, and $c_*$.
\end{theorem}

\begin{proof}
Choose $t_0\in[0,2\pi)$ with $\Omega(t_0)\ne0$, and set
\[
  \xi=\frac{\overline{\Omega(t_0)}}{|\Omega(t_0)|}.
\]
By continuity on the circle, there is a number $\delta\in(0,\pi)$ such that
\begin{equation}\label{eq:sector-arc}
  \operatorname{Re}\bigl(\xi\Omega(t)\bigr)
  \ge\frac12|\Omega(t_0)|
  \qquad
  \text{if }\distT(t,t_0)\le\delta.
\end{equation}
Put
\[
  Y_N=\kappa\lambda^\nu N^{\nu(a+1)}
\]
and choose
\begin{equation}\label{eq:j-choice}
  j_N=\left\lceil\frac{Y_N-t_0}{2\pi}\right\rceil,
  \qquad
  x_N=\frac1N
  \left(\frac{2\pi j_N+t_0}{\kappa}\right)^{1/\nu}.
\end{equation}
Since $2\pi j_N+t_0=Y_N+O(1)$ and $Y_N\to\infty$, the mean-value theorem applied to $u\mapsto u^{1/\nu}$ gives \eqref{eq:sector-x}.  Moreover,
\begin{equation}\label{eq:exact-resonance}
  \kappa(Nx_N)^\nu=2\pi j_N+t_0
  \equiv t_0\pmod{2\pi}.
\end{equation}

Let
\[
  \Phi_N(u)=\kappa(ux_N)^\nu.
\]
For $u\in[N/2,N]$, equation \eqref{eq:sector-x} gives
\begin{equation}\label{eq:general-derivative}
  0<\Phi_N'(u)
  =\kappa\nu x_N^\nu u^{\nu-1}
  \le C N^{\nu(a+1)-1}
\end{equation}
for all sufficiently large $N$, where $C$ depends only on the fixed data.  Choose $\eta>0$ so small that $2C\eta\le\delta$, and set
\begin{equation}\label{eq:H-choice-general}
  H_N=\left\lfloor\eta N^{1-\nu(a+1)}\right\rfloor.
\end{equation}
The two inequalities in \eqref{eq:a-general-range} give
$0<1-\nu(a+1)<1$.  Hence $H_N\to\infty$ and $H_N=o(N)$, so $1\le H_N<N/2$ for all sufficiently large $N$.  If $N-H_N<n\le N$, then by \eqref{eq:general-derivative},
\[
  |\Phi_N(N)-\Phi_N(n)|
  \le C H_NN^{\nu(a+1)-1}
  \le\delta.
\]
Together with \eqref{eq:exact-resonance} and \eqref{eq:sector-arc}, this yields
\begin{equation}\label{eq:sector-positive}
  \operatorname{Re}\bigl(\xi\Omega(\Phi_N(n))\bigr)
  \ge\frac12|\Omega(t_0)|
  \qquad(N-H_N<n\le N).
\end{equation}
On this block, $n\asymp N$, and therefore $n^{-\beta}\asymp_{\beta}N^{-\beta}$.  Using \eqref{eq:coeff-lower} and \eqref{eq:sector-positive}, we obtain
\begin{align*}
  &\bigl|S_{\Omega}(x_N,N)-S_{\Omega}(x_N,N-H_N)\bigr|\\
  &\quad\ge
  \operatorname{Re}\!\left(
    \xi\sum_{N-H_N<n\le N}
    c_n n^{-\beta}\Omega(\Phi_N(n))
  \right)\\
  &\quad\gg H_NN^{-\beta}
  \gg N^{1-\beta-\nu(a+1)}.
\end{align*}
This proves \eqref{eq:sector-block}.  Finally,
\[
  |S_{\Omega}(x_N,N)-S_{\Omega}(x_N,N-H_N)|
  \le |S_{\Omega}(x_N,N)|+|S_{\Omega}(x_N,N-H_N)|,
\]
which implies \eqref{eq:sector-max}.
\end{proof}

\begin{remark}\label{rem:natural-scale}
The length in \eqref{eq:sector-H} is the reciprocal first-derivative scale.  Indeed,
\[
  \Phi_N'(N)\asymp N^{\nu(a+1)-1},
  \qquad
  H_N\asymp\frac1{\Phi_N'(N)}.
\]
\end{remark}

This mechanism may be viewed as complementary to the classical first-derivative principle, see, for example \cite{AriasDeReyna}.  The proof is deterministic and local.  Classical and recent divisor-type $\Omega$-results instead seek values of the external variable for which many terms of a Voronoi expansion
reinforce one another; see \cite{Soundararajan,Sourmelidis,Mahatab,Lamzouri,Jaasaari}. 

Ren and Ye study twists $e(\alpha n^\beta)$ for
which special values of $\alpha$ match the archimedean phase produced
by Vorono\"i summation, thereby creating a stationary point and a main
term \cite{RenYe2010,RenYe2015}.  Likewise, the standard twist studied
by Kaczorowski and Perelli has a distinguished resonance spectrum
reflected in its analytic continuation and polar structure
\cite{KaczorowskiPerelli}. 

\begin{remark}
\label{rem:CS-scope}
In the self-dual archimedean setting considered by Clozel and Sarnak,
the positive- and negative-height contributions lead to the
$2\pi$-periodic profile
\begin{equation}\label{eq:CS-amplitude}
  \Omega_{\omega_1,\omega_2}(t)
  :=
  \omega_1\e^{-i\pi/4-it}
  +
  \omega_2\e^{i\pi/4+it},
  \qquad
  (\omega_1,\omega_2)\ne(0,0);
\end{equation}
see \cite[Hypothesis~2.1, Lemma~2.1, and
(2.30)]{ClozelSarnak}.  Since the Fourier modes $e^{-it}$ and $e^{it}$
are linearly independent,
\[
  \Omega_{\omega_1,\omega_2}\equiv0
  \quad\Longleftrightarrow\quad
  \omega_1=\omega_2=0.
\]
Thus \eqref{eq:CS-amplitude} is a nonzero continuous periodic profile,
and Theorem~\ref{thm:sector} applies directly to
\[
  \Bcal_m^{\omega_1,\omega_2}(x,M)
  :=
  \sum_{n\le M}
  d_m(n)n^{-\beta_m}
  \Omega_{\omega_1,\omega_2}
  \!\left(2\pi m(nx)^{1/m}\right).
\]
Consequently, for every $\lambda>0$ and $0<a<m-1$, all conclusions of
Theorem~\ref{thm:main} remain valid with $\Bcal_m$ replaced by
$\Bcal_m^{\omega_1,\omega_2}$, with constants allowed to depend on
$\omega_1$ and $\omega_2$.

This is only an amplitude-level statement.  It does not assert that an
arbitrary pair $(\omega_1,\omega_2)$ occurs together with the
coefficients $d_m(n)$ in an automorphic summation formula.  Moreover,
the application in \cite{ClozelSarnak} assumes at most a simple pole,
and therefore is not applied here to $\zeta(s)^m$, which has a pole of
order $m$.  Finally, taking $a=1$ and
$\lambda>(2\pi)^m$ gives
\[
  x_N\sim\lambda N,
  \qquad
  1\le N\le(2\pi)^{-m}x_N
\]
for all sufficiently large $N$.  This merely checks compatibility
with the numerical subrange in
\cite[(2.30)]{ClozelSarnak}; it is not an invocation of their
summation formula for the zeta-product datum.
\end{remark}

\section{The zeta-product datum and the sharp conjecture}\label{sec:zeta}

\subsection{Friedlander--Iwaniec normalization}

As in \cite{FriedlanderIwaniec}, we begin with coefficient bounds
\[
  a(n),b(n)\ll_\varepsilon n^\varepsilon
  \qquad(\varepsilon>0),
\]
and Dirichlet series
\[
  A(s)=\sum_{n\ge1}a(n)n^{-s},
  \qquad
  B(s)=\sum_{n\ge1}b(n)n^{-s}.
\]
They assume that $A(s)$ continues meromorphically to $\C$, with no possible pole except one of finite order at $s=1$, that $\gamma(s)$ is holomorphic for $\operatorname{Re}s>1/2$, and that, for $\operatorname{Re}s>1$,
\begin{equation}\label{eq:FI-FE}
  A(1-s)=w\gamma(s)B(s),
  \qquad |w|=1.
\end{equation}
For the gamma-factor class used in their main theorem,
\begin{equation}\label{eq:FI-gamma}
  \gamma(s)
  =(\pi^{-m}D)^{s-1/2}
  \prod_{j=1}^{m}
  \frac{\Gamma((s+\kappa_j)/2)}
       {\Gamma((1-s+\kappa_j)/2)},
  \qquad
  \operatorname{Re}\kappa_j\ge0.
\end{equation}
Here $m$ is the degree, $D\in\mathbb N$ is the conductor, $w$ is the root number, and the $\kappa_j$ are the spectral parameters.  Friedlander and Iwaniec set
\[
  k=\operatorname{Re}\sum_{j=1}^{m}\kappa_j\in\Z,
  \qquad
  \ell=m-3-2k,
\]
and call $k$ the weight (see \cite[(1.4)--(1.8)]{FriedlanderIwaniec}).  Their initial analytic framework permits a pole at $s=1$ of arbitrary finite order and explicitly includes products of zeta and Dirichlet $L$-functions \cite[pp.~494--495]{FriedlanderIwaniec}.

\begin{proposition}[The datum $\zeta(s)^m$]\label{prop:zeta-datum}
For every integer $m\ge2$, the pair
\[
  A(s)=B(s)=\zeta(s)^m
\]
satisfies \eqref{eq:FI-FE}--\eqref{eq:FI-gamma} with
\[
  D=1,
  \qquad w=1,
  \qquad \kappa_1=\cdots=\kappa_m=0,
  \qquad k=0,
\]
and with dual coefficients $b(n)=d_m(n)$.  Thus $\ell=m-3$, and the corresponding sharp sum is exactly \eqref{eq:Bm-intro}.  Moreover,
\begin{equation}\label{eq:dm-lower}
  d_m(n)\ge1\qquad(n\ge1).
\end{equation}
\end{proposition}

\begin{proof}
The Dirichlet series for $\zeta(s)^m$ converges absolutely in $\operatorname{Re}s>1$.  The Riemann zeta function extends meromorphically to $\C$, with a unique pole, simple and located at $s=1$, and satisfies
\[
  \zeta(1-s)
  =\pi^{1/2-s}
  \frac{\Gamma(s/2)}{\Gamma((1-s)/2)}\zeta(s).
\]
See, for example, Titchmarsh \cite[Theorem~2.1]{Titchmarsh}.  It follows that $\zeta(s)^m$ has a pole of order $m$ at $s=1$ and no other pole.  Raising the displayed functional equation to the $m$th power gives \eqref{eq:FI-FE}--\eqref{eq:FI-gamma} with
\[
  D=1,\qquad w=1,\qquad \kappa_1=\cdots=\kappa_m=0.
\]
The boundary value $\operatorname{Re}\kappa_j=0$ is explicitly permitted in \cite[(1.4)]{FriedlanderIwaniec}.  For these parameters, the quotient in \eqref{eq:FI-gamma} is holomorphic in $\operatorname{Re}s>1/2$: the numerator $\Gamma(s/2)$ has poles only at $s=0,-2,-4,\ldots$, while the reciprocal of the denominator gamma function is entire.  The parameter array is real, so
\[
  \gamma(\overline{s})=\overline{\gamma(s)}.
\]
Consequently the leading constants in the positive- and negative-height Stirling expansions are conjugate, and the two modes combine into the cosine in \eqref{eq:FI-B-intro}.  Thus $k=0$, $\ell=m-3$, and the displayed sharp sum is \eqref{eq:Bm-intro}.

The coefficients are the integers $d_m(n)$ defined in \eqref{eq:dm-definition-intro}.  They are multiplicative, and for every prime power $p^r$,
\[
  d_m(p^r)=\#\{(r_1,\ldots,r_m)\in\mathbb Z_{\ge0}^m:
                    r_1+\cdots+r_m=r\}
           =\binom{r+m-1}{m-1}.
\]
This proves \eqref{eq:dm-lower}.  For completeness, fix $\varepsilon>0$.  The number of weak compositions above is at most $m^r$, because a word of length $r$ on an $m$-letter alphabet maps surjectively to its vector of letter multiplicities.  Hence
\[
  d_m(p^r)\le m^r\le p^{\varepsilon r}
\]
for every prime $p\ge m^{1/\varepsilon}$.  For each of the finitely many smaller primes,
\[
  \sup_{r\ge0}\binom{r+m-1}{m-1}p^{-\varepsilon r}<\infty,
\]
since the binomial coefficient grows polynomially in $r$ whereas $p^{\varepsilon r}$ grows exponentially.  Multiplicativity now gives the required coefficient bound
\[
  d_m(n)\ll_{m,\varepsilon}n^\varepsilon.
\]
\end{proof}

\subsection{Proof of the endpoint lower bound}

\begin{proof}[Proof of Theorem~\ref{thm:main}]
Apply Theorem~\ref{thm:sector} with
\[
  \nu=\frac1m,
  \qquad
  \beta=\beta_m=\frac{m+1}{2m},
  \qquad
  \kappa=2\pi m,
  \qquad
  c_n=d_m(n),
\]
and
\[
  \Omega_m(t)=\cos(t+\phi_m).
\]
The hypothesis $0<a<m-1$ ensures
$-1<a<\nu^{-1}-1$, as required in \eqref{eq:a-general-range}.  By \eqref{eq:dm-lower}, the coefficient hypothesis holds with $c_*=1$.  The exponent in \eqref{eq:sector-block} is
\[
  1-\beta_m-\frac{a+1}{m}
  =\frac{m-3-2a}{2m}
  =\Delta_m(a).
\]
Equations \eqref{eq:main-x}--\eqref{eq:main-block} now follow from Theorem~\ref{thm:sector}.
\end{proof}

The resonant parameter may be made completely explicit.  With
\[
  \alpha=\frac{a+1}{m},
\]
set
\begin{equation}\label{eq:explicit-x}
  J_N=
  \left\lceil
    m\lambda^{1/m}N^\alpha+\frac{m-3}{8}
  \right\rceil,
  \qquad
  x_N=\frac1N
  \left(\frac{J_N-(m-3)/8}{m}\right)^m.
\end{equation}
Then \eqref{eq:main-x} holds and
\begin{equation}\label{eq:exact-cosine-maximum}
  2\pi m(Nx_N)^{1/m}+\phi_m=2\pi J_N.
\end{equation}
Thus the last term is placed exactly at a maximum of the cosine.

\subsection{Proof of Corollary~\ref{cor:consequences}}

\begin{proof}
Suppose first that \eqref{eq:a-power} holds.  Then $\Delta_m(a)>0$.  If \eqref{eq:FI-conj-intro} held for $B(s)=\zeta(s)^m$, it would hold at both cutoffs $M=N-H_N$ and $M=N$.  Since $M\asymp N$ and $x_N\asymp N^a$, this would give
\[
  |\Bcal_m(x_N,M)|\ll_{m,\eps}N^{(1+a)\eps}
\]
for both cutoffs.  Choosing $\eps>0$ so that
\[
  (1+a)\eps<\Delta_m(a)
\]
contradicts \eqref{eq:main-block}.

In fact, if, in addition,
\begin{equation}\label{eq:a-contour}
  a>\frac1{m-1},
\end{equation}
then both cutoffs $N-H_N$ and $N$ belong to the numerical cutoff range generated by a contour height $1\le T\le x_N$ in the Friedlander--Iwaniec
proof. Indeed, for either cutoff $M\in\{N-H_N,N\}$, one has $M\asymp N$ and
\[
  \frac{x_N^{m-1}}{(2\pi)^mM}
  \asymp N^{a(m-1)-1}\longrightarrow\infty.
\]
Hence
\begin{equation}\label{eq:contour-upper}
  M\le\frac{x_N^{m-1}}{(2\pi)^m}
\end{equation}
for large $N$.  Since $x_N\to\infty$ and $M\ge1$, the associated height for $D=1$,
\[
  T_M=2\pi(Mx_N)^{1/m},
\]
satisfies $1\le T_M\le x_N$.  Thus both cutoffs are in the numerical range generated by the contour parameter.  The interval
\[
  \frac1{m-1}<a<\frac{m-3}{2}
\]
is nonempty for every integer $m\ge4$, because
\[
  \frac{m-3}{2}-\frac1{m-1}
  =\frac{m^2-4m+1}{2(m-1)}>0.
\]
This proves part~\textup{(i)}.

For part~\textup{(ii)}, take $a=1$ and $\lambda>1$.  Then
\[
  \Delta_m(1)=\frac{m-5}{2m},
\]
and \eqref{eq:main-x} implies $N\le x_N$ for all sufficiently large $N$.  The lower cutoff is positive because $H_N=o(N)$.  The exponent is positive exactly when $m\ge6$, and the same comparison with $N^{2\eps}$ contradicts \eqref{eq:FI-conj-intro}.

For part~\textup{(iii)}, take
\[
  a=\frac2{m-1},
  \qquad
  \lambda=1.
\]
Then \eqref{eq:main-x} gives
\[
  N=x_N^{(m-1)/2}(1+o(1)),
\]
and
\begin{align*}
  \Delta_m\!\left(\frac2{m-1}\right)
  &=\frac1{2m}\left(m-3-\frac4{m-1}\right)\\
  &=\frac{m^2-4m-1}{2m(m-1)}.
\end{align*}
This number is positive exactly when $m>2+\sqrt5$, hence exactly for integral $m\ge5$.
\end{proof}

\end{document}